\documentclass[a4paper,11pt,reqno,twoside]{article}
\usepackage{amssymb}
\usepackage[frenchb,english]{babel}
\usepackage{amscd}
\usepackage{amsmath,amsthm,amsfonts,amssymb,graphicx}
\usepackage[colorlinks,linkcolor=blue,anchorcolor=blue,citecolor=green]{hyperref}
\usepackage{mathrsfs}
\usepackage{fancyhdr}
\usepackage{subfigure}
\usepackage{bm}
\usepackage{multicol}
\usepackage{picins}
\usepackage{abstract}

\begin{document}
\theoremstyle{plain}
\newtheorem{Definition}{Definition}[section]
\newtheorem{Proposition}{Proposition}[section]
\newtheorem{Property}{Property}[section]
\newtheorem{Theorem}{Theorem}[section]
\newtheorem{Lemma}[Theorem]{\hspace{0em}\bf{Lemma}}
\newtheorem{Corollary}[Theorem]{Corollary}
\newtheorem{Remark}{Remark}[section]
\newtheorem{Example}{Example}[section]

\baselineskip 15pt

\noindent  {\LARGE Balanced  metrics on the Fock-Bargmann-Hartogs domains}\\\\

\noindent\text{Enchao Bi$^1$, \; Zhiming Feng$^2$ \; \& \; Zhenhan T{u$^1$}$^{*}$ }\\

\noindent\small {${}^1$School of Mathematics and Statistics, Wuhan
University, Wuhan, Hubei 430072, P.R. China} \\

\noindent\small {${}^2$School of Mathematical and Information Sciences, Leshan Normal University, Leshan, Sichuan 614000, P.R. China } \\

\noindent\text{Email:  bienchao@whu.edu.cn (E. Bi), \; fengzm2008@163.com (Z. Feng), \; zhhtu.math@whu.edu.cn (Z. Tu)}

\renewcommand{\thefootnote}{{}}
\footnote{\hskip -16pt {$^{*}$Corresponding author. \\}}
\\

\normalsize \noindent\textbf{Abstract}\quad {The Fock-Bargmann-Hartogs domain $D_{n,m}(\mu)$ ($\mu>0$)
 in $\mathbb{C}^{n+m}$ is defined by the inequality $\|w\|^2<e^{-\mu\|z\|^2},$  where $(z,w)\in \mathbb{C}^n\times \mathbb{C}^m$, which is an unbounded non-hyperbolic domain
 in $\mathbb{C}^{n+m}$.
This paper introduces a K\"{a}hler metric $\alpha g(\mu;\nu)$ $(\alpha>0)$ on $D_{n,m}(\mu)$, where $g(\mu;\nu)$ is the K\"{a}hler metric associated with the K\"{a}hler potential $\Phi(z,w):=\mu\nu{\Vert z\Vert}^{2}-\ln(e^{-\mu{\Vert z\Vert}^{2}}-\Vert w\Vert^2)$ ($\nu>-1$) on $D_{n,m}(\mu)$.
The purpose of this paper is twofold. Firstly, we obtain an explicit formula for the Bergman kernel of the weighted Hilbert space  of square integrable
holomorphic functions on $(D_{n,m}(\mu), g(\mu;\nu))$ with the weight $\exp\{-\alpha \Phi\}$ for $\alpha>0$.
Secondly, using the explicit expression of the Bergman kernel, we obtain the necessary and sufficient condition for the metric $\alpha g(\mu;\nu)$ $(\alpha>0)$ on the domain
$D_{n,m}(\mu)$ to be a balanced metric.
So we obtain
the existence of balanced metrics for a class of Fock-Bargmann-Hartogs domains.

\vskip 10pt

\vskip 10pt

  \textperiodcentered \, 32M15  \textperiodcentered \, 32Q15

\setlength{\oddsidemargin}{-.5cm}  
\setlength{\evensidemargin}{\oddsidemargin}
\pagenumbering{arabic}
\renewcommand{\theequation}
{\arabic{section}.\arabic{equation}}

\setcounter{equation}{0}

\section{{Introduction}}
 Assume that $M$ is a complex manifold and $\varphi$
is a strictly plurisubharmonic function on $M$. Let $g$ be a
K\"{a}hler metric on $M$ associated with the K\"{a}hler form
$\omega=\frac{\sqrt{-1}}{2\pi}\partial\overline{\partial}\varphi$.
For $\alpha>0$, let $\mathcal{H}_{\alpha}$ be the weighted Hilbert
space of square integrable holomorphic functions on $(M, g)$ with
the weight $\exp\{-\alpha \varphi\}$, that is,
$$\mathcal{H}_{\alpha}:=\left\{ f\in \textmd{Hol}(M): \int_{M}{\vert f\vert}^2\exp\{-\alpha \varphi\}\frac{\omega^n}{n!}<+\infty\right\},$$
where $\textmd{Hol}(M)$ is the space of holomorphic functions
on $M$. Let $K_{\alpha}(z,\overline{z})$ be the Bergman kernel
(namely, the reproducing kernel) of the Hilbert space
$\mathcal{H}_{\alpha}$ if $\mathcal{H}_{\alpha}\neq \{0\}$. The
Rawnsley's $\varepsilon$-function on $M$ associated with the metric
$g$  is defined by
\begin{equation}\label{eq1.1}
 \varepsilon_{(\alpha,g)}(z):=\exp\{-\alpha \varphi(z)\}K_{\alpha}(z,\overline{z}),\;\; z\in M.
\end{equation}
Note the Rawnsley's $\varepsilon$-function depends only on the
metric $g$ and not on the choice of the K\"{a}hler potential
$\varphi$. The asymptotics of the Rawnsley's $\varepsilon$-function $
\varepsilon_{\alpha}$ was expressed in terms of the parameter
$\alpha$ for compact manifolds by Catlin \cite{Cat} and Zelditch
\cite{Zeld} (for $\alpha\in \mathbb{N}$) and for non-compact
manifolds by Ma-Marinescu \cite{MM07,MM08, MM12}. In some particular case
it was also proved by Engli\v{s} \cite{E1,E2}.

\begin{Definition} The metric $\alpha g$ on $M$ is balanced  if the Rawnsley's
$\varepsilon$-function $\varepsilon_{(\alpha,g)}(z)$ $(z\in M)$ is a
positive constant on $M$.
\end{Definition}

The definition of balanced metrics was originally given by Donaldson
\cite{Donaldson} in the case of a compact polarized K\"{a}hler
manifold $(M,g)$ in 2001, who also established the existence of such
metrics on any (compact) projective K\"{a}hler manifold with
constant scalar curvature. The definition of balanced metrics was
generalized in Arezzo-Loi
 \cite{Are-Loi} and Engli\v{s} \cite{E3} to the noncompact case. We give only
the definition for those K\"{a}hler metrics which admit globally
defined potentials in this paper.

Balanced metrics play a fundamental role in the quantization  of a
K\"{a}hler manifold (e.g., see Berezin \cite{Berezin},
Cahen-Gutt-Rawnsley \cite{CGR} and Engli\v{s} \cite{E0}). It also related to the Bergman kernel expansion (e.g.,
see Loi \cite{Lo} and references therein).
For the study of the balanced metrics,  see also Feng-Tu \cite{FT, FT2}, Loi \cite{Loi2}, Loi-Mossa \cite{Loi-Mossa}, Loi-Zedda \cite{Loi-Zed, LZ},
Loi-Zedda-Zuddas \cite{Loi-Zed-Zud} and Zedda \cite{Zed}.

\vskip 5pt
 The Cartan-Hartogs domains and the Fock-Bargmann-Hartogs domains are two kinds of typical Hartogs domains (e.g., see Kim-Yamamori \cite{KY}).
 The Cartan-Hartogs domains are some Hartogs domains over bounded symmetric domains and there are many researches about the balanced metrics on the  Cartan-Hartogs domains
 (e.g., see Feng-Tu \cite{FT, FT2},  Loi-Zedda \cite{LZ} and Zedda \cite{Zed}).
The Fock-Bargmann-Hartogs domains are some Hartogs domains over $\mathbb{C}^{n}$, and, however, very little seems to be known about the existence of balanced metrics on
the Fock-Bargmann-Hartogs domains. In this paper we will obtain the existence of balanced metrics on a class of the Fock-Bargmann-Hartogs domains.
Therefore, together with Feng-Tu \cite{FT2},  the result of the present paper  represents an  example of balanced metric  on a complex domain which is not homogeneous.

\vskip 5pt
For a given positive real number $\mu$, the Fock-Bargmann-Hartogs domain $D_{n,m}(\mu)$ is a Hartogs domain over $\mathbb{C}^{n}$ defined by
\begin{equation}
D_{n,m}(\mu):=\big\{(z,w)\in \mathbb{C}^{n+m}: \Vert w\Vert^{2}<e^{-\mu{\Vert z\Vert}^{2}}\big\},
\end{equation}
where $\Vert\cdot\Vert$ is the standard Hermitian norm. The
Fock-Bargmann-Hartogs domains $D_{n,m}(\mu)$ are strongly
pseudoconvex, nonhomogeneous unbounded domains in $\mathbb{C}^{n+m}$ with smooth real-analytic boundary. We note that each
$D_{n,m}(\mu)$ contains $\{(z, 0)\in
\mathbb{C}^n\times\mathbb{C}^m\} \cong \mathbb{C}^n$. Thus each
$D_{n,m}(\mu)$ is not hyperbolic in the sense of Kobayashi and
$D_{n,m}(\mu)$ can not be biholomorphic to any bounded domain in
$\mathbb{C}^{n+m}$. Therefore, each Fock-Bargmann-Hartogs domain
$D_{n,m}(\mu)$ is an unbounded non-hyperbolic domain in
$\mathbb{C}^{n+m}.$  For the general reference of the Fock-Bargmann-Hartogs domain, see Kim-Ninh-Yamamori \cite{Kim}, Tu-Wang \cite {TW} and references therein.

Now we introduce a new K\"{a}hler metric $g(\mu;\nu)$ on $D_{n,m}(\mu)$. Let $\nu>-1$, and define the strictly plurisubharmonic function $\Phi(z,w)$ on the Fock-Bargmann-Hartogs domain $D_{n,m}(\mu)$ as follows
\begin{equation}\label{eq1.3}
\Phi(z,w):=\nu\mu{\Vert z\Vert}^{2}
-\ln(e^{-\mu{\Vert z\Vert}^{2}}-\Vert w\Vert^2).
\end{equation}
The K\"{a}hler form $\omega(\mu;\nu)$ on $D_{n,m}(\mu)$ is given by
\begin{equation*}
\omega(\mu;\nu):=\frac{\sqrt{-1}}{2\pi}\partial\overline{\partial}\Phi.
\end{equation*}
Hence the K\"{a}hler metric $g(\mu;\nu)$ on $D_{n,m}(\mu)$ associated with $\omega(\mu;\nu)$ can be expressed by
\begin{equation*}
g(\mu;\nu)_{i\overline{j}}=\frac{\partial^2\Phi}{\partial Z_{i}\partial\overline{Z_{j}}}\;\; (1\leq i,j\leq n+m),
\end{equation*}
where $(Z_{1},\cdots,Z_{n+m})=(z,w)$.

In the case of $n=1$ and $\nu=0$, the K\"{a}hler metric $g(\mu;\nu)$ reduces to the canonical metric in Loi-Zedda \cite{Loi-Zed} with $F(x)=\exp\{-\mu x\}$.
But, we will see that, in the case of  the Fock-Bargmann-Hartogs domain $D_{n,m}(\mu)$,  any metric $\alpha g(\mu;0)\;(\alpha>0)$ is not balanced by Th. 1.1 of the present paper.

The main result of the paper is the following.

\begin{Theorem}\label{Thm1.1}
Let the Fock-Bargmann-Hartogs domain $D_{n,m}(\mu)$ be endowed with the K\"{a}hler metric $g(\mu;\nu)$. Then the metric $\alpha g(\mu;\nu)$ on $D_{n,m}(\mu)$ is balanced if and only
if $\alpha>m+n,$  $n=1$ and $\nu=-\frac{1}{m+1}.$
\end{Theorem}

The paper is organized as follows. In Section 2, we give an explicit formula for the Bergman kernel of the weighted Hilbert space  of square integrable
holomorphic functions on $(D_{n,m}(\mu), g(\mu;\nu))$ with the weight $\exp\{-\alpha \Phi\}$ for $\alpha>0$, and
thus obtain the explicit expression of the Rawnsley's $\varepsilon$-function of $D_{n,m}(\mu)$ with respect to the metric $g(\mu;\nu)$.
In Section 3, using the expression of the Rawnsley's $\varepsilon$-function, we prove Theorem \ref{Thm1.1}.

\setcounter{equation}{0}
\section{The Rawnsley's $\varepsilon$-function for $D_{n,m}(\mu)$  with the metric $g(\mu;\nu)$}

We firstly give the explicit description of automorphism of the Fock-Bargmann-Hartogs domain $D_{n,m}(\mu)$ as follows.

\begin{Lemma}[Kim-Ninh-Yamamori \cite{Kim}]
The automorphism group $\mathrm{Aut}(D_{n,m}(\mu))$ is generated by all the following automorphisms of $D_{n,m}(\mu)$:\\
\begin{equation}
\varphi_{U}:\;(z,w)\rightarrow (Uz,w),\;U\in\mathcal{U}(n);
\end{equation}
\begin{equation}
\varphi_{V}:\;(z,w)\rightarrow(z,Vw),\;V\in \mathcal{U}(m);
\end{equation}
\begin{equation}\label{eq2.3}
\varphi_{a}:\;(z,w)\rightarrow (z-a,e^{\mu\langle z,a\rangle-\frac{\mu}{2}{\Vert a\Vert}^{2}}w),\;a\in\mathbb{C}^{n},
\end{equation}
where $\mathcal{U}(n)$ denotes the set of the $n\times n$ unitary matrices.
\end{Lemma}

\begin{Lemma}
Let $F(z,w):=(z-a,e^{\mu\langle z,a\rangle-\frac{\mu}{2}{\Vert a\Vert}^{2}}w)\;(a\in\mathbb{C}^{n})$ (i.e., an automorphism of $D_{n,m}(\mu)$ of the form \eqref{eq2.3}). Then we have
\begin{equation}\label{eq2.4}
\partial\overline{\partial}\Phi(F)=\partial\overline{\partial}\Phi,
\end{equation}
and moreover
\begin{equation}\label{eq2.5}
\det\big(\frac{\partial F_{j}}{\partial Z_{i}}\big)(z_{0},w_{0})=e^{m\mu\langle z_{0},a\rangle-\frac{m\mu}{2}{\Vert a\Vert}^{2}},
\end{equation}
where $\Phi$ is defined by \eqref{eq1.3} and $(Z_{1},\cdots,Z_{m+n})=(z,w)\in D_{n,m}(\mu)$.
\end{Lemma}

\begin{proof}[Proof]
By the definition of $F$, we have
\begin{equation*}
\begin{aligned}
&\Phi(F)\\
&=\mu\nu\Vert z-a\Vert^{2}-\ln e^{2\mu\mathrm{Re}\langle z,a\rangle-\mu\Vert a\Vert^{2}}(e^{-\mu{\Vert z\Vert}^{2}}-\Vert w\Vert^2)\\
&=\mu\nu\Vert z\Vert^{2}-\mu(\nu+1)(2\mathrm{Re}\langle z,a\rangle-\Vert a\Vert^{2})-\ln(e^{-\mu{\Vert z\Vert}^{2}}-\Vert w\Vert^2).\\
\end{aligned}
\end{equation*}
It follows that
$$\partial\overline{\partial}\Phi(F)=\partial\overline{\partial}\Phi.$$

On the other hand, it is easy to see
\begin{equation*}
\big(\frac{\partial F_{j}}{\partial Z_{i}}\big)(z_{0},w_{0})
=\left(\begin{array}{cc}
I_{n\times n}&0\\
\ast&e^{\mu\langle z_{0},a\rangle-\frac{\mu}{2}{\Vert a\Vert}^{2}}I_{m\times m}\\
\end{array}
\right),
\end{equation*}
where $I_{n\times n}$ and $I_{m\times m}$ denote the $n\times n$ and $m\times m$ diagonal matrices with the diagonal elements $1$, respectively.
 It implies
$$\det\big(\frac{\partial F_{j}}{\partial Z_{i}}\big)(z_{0},w_{0})=e^{m\mu\langle z_{0},a\rangle-\frac{m\mu}{2}{\Vert a\Vert}^{2}}.$$
The proof is finished.
\end{proof}

\begin{Lemma}\label{Le2.3}
Let $\Phi$ be defined by \eqref{eq1.3}. Then we have
\begin{equation}\label{eq2.6}
\det\bigg(\frac{\partial^{2}\Phi}{\partial Z_{i}\partial\overline{Z_{j}}}\bigg)(z,w)=\frac{\mu^{n}\big[\nu+
(1-{\Vert \widetilde{w}\Vert}^{2})^{-1}\big]^{n}}{(1-{\Vert \widetilde{w}\Vert}^{2})^{m+1}}e^{m\mu\Vert z\Vert^2},
\end{equation}
where $(Z_{1},\cdots,Z_{m+n})=(z,w)\in D_{n,m}(\mu)$ and $\widetilde{w}$ is defined by
\begin{equation}\label{eq2.7}
\widetilde{w}:=e^{\frac{\mu}{2}\Vert z\Vert^2}w.
\end{equation}
\end{Lemma}

\begin{proof}[Proof] For any $(z_{0},w_{0})\in D_{n,m}(\mu)$, consider the automorphism $F$ of $D_{n,m}(\mu)$ as follows:
$$F(z,w):=(z-z_{0},e^{\mu\langle z,z_{0}\rangle-\frac{\mu}{2}{\Vert z_{0}\Vert}^{2}}w).$$
Thus, $F(z_{0},w_{0})=(0,\widetilde{w_{0}})$, where $\widetilde{w_{0}}=e^{\frac{\mu}{2}\Vert z_{0}\Vert^2}w_{0}$. Due to \eqref{eq2.4}, we get
\begin{equation}\label{eq2.8}
\det\bigg(\frac{\partial^{2}\Phi(F)}{\partial Z_{i}\partial\overline{Z_{j}}}\bigg)(z_{0},w_{0})=\det\bigg(\frac{\partial^{2}\Phi}{\partial Z_{i}\partial\overline{Z_{j}}}\bigg)(z_{0},w_{0}).
\end{equation}
Since
\begin{equation}
\bigg(\frac{\partial^{2}\Phi(F)}{\partial Z_{i}\partial\overline{Z_{j}}}\bigg)(z_{0},w_{0})=\big(\frac{\partial F_{j}}{\partial Z_{i}}\big)
\bigg(\frac{\partial^{2}\Phi}{\partial Z_{i}\partial\overline{Z_{j}}}\bigg)(F(z_{0},w_{0}))\big(\frac{\partial \overline{F_{i}}}{\partial \overline{Z_{j}}}\big),
\end{equation}
by combining with \eqref{eq2.8}, we have
\begin{equation}\label{eq2.10}
\det\bigg(\frac{\partial^{2}\Phi}{\partial Z_{i}\partial\overline{Z_{j}}}\bigg)(z_{0},w_{0})=\bigg|\det\big(\frac{\partial F_{j}}{\partial Z_{i}}\big)(z_{0},w_{0})\bigg|^{2}\det\bigg(\frac{\partial^{2}\Phi}{\partial
Z_{i}\partial\overline{Z_{j}}}\bigg)(F(z_{0},w_{0})).
\end{equation}

Note the formula \eqref{eq2.5} implies
\begin{equation}\label{eq2.11}
\bigg|\det\big(\frac{\partial F_{j}}{\partial Z_{i}}\big)(z_{0},w_{0})\bigg|^{2}=e^{m\mu\Vert z_{0}\Vert^2}.
\end{equation}

By \eqref{eq1.3}, we have
\begin{equation*}
\begin{aligned}
&\frac{\partial^{2}\Phi}{\partial z_{i}\partial\overline{z_{j}}}(z,w)&=&\mu\nu\delta_{ij}+\frac{1}{(1-\Vert \widetilde{w}\Vert^{2})^{2}}\big[\mu\delta_{ij}(1-\Vert \widetilde{w}\Vert^{2})+\mu^{2}\overline{z_{i}}z_{j}
\big]\Vert\widetilde{w}\Vert^{2},\\
&\frac{\partial^{2}\Phi}{\partial z_{i}\partial\overline{w_{j}}}(z,w)&=&\frac{\mu e^{{\frac{\mu}{2}}\Vert z \Vert^{2}}\overline{z_{i}}\widetilde{w}_{j}}{(1-\Vert \widetilde{w}\Vert^{2})^{2}},\\
&\frac{\partial^{2}\Phi}{\partial w_{i}\partial\overline{w_{j}}}(z,w)&=&\frac{e^{\mu\Vert z \Vert^{2}}}{(1-\Vert \widetilde{w}\Vert^{2})^{2}}\big[\delta_{ij}(1-\Vert
\widetilde{w}\Vert^{2})+\overline{\widetilde{w_{i}}}\widetilde{w_{j}}\big].\\
\end{aligned}
\end{equation*}
In particular, when we evaluate at $(0,\widetilde{w_{0}})$, we obtain
\begin{equation*}
\bigg(\frac{\partial^{2}\Phi}{\partial Z_{i}\partial\overline{Z_{j}}}\bigg)(0,\widetilde{w_{0}})=
\left(\begin{array}{cc}
\mu\big[\nu+(1-\Vert\widetilde{w_{0}} \Vert^2)^{-1}\big]I_{n\times n}&0\\
0&\frac{1}{1-\Vert \widetilde{w_{0}}\Vert^2}I_{m\times m}+\frac{1}{(1-\Vert \widetilde{w_{0}}\Vert^{2})^{2}}\overline{\widetilde{w_{0}}}^{t}\widetilde{w_{0}}\\
\end{array}
\right),
\end{equation*}
which implies
\begin{equation}\label{eq2.12}
\det\bigg(\frac{\partial^{2}\Phi}{\partial Z_{i}\partial\overline{Z_{j}}}\bigg)(0,\widetilde{w_{0}})=\frac{\mu^{n}\big[\nu+(1-\Vert\widetilde{w_{0}} \Vert^2)^{-1}\big]^{n}}{(1-\Vert \widetilde{w_{0}}\Vert^2)^{m+1}}.
\end{equation}
Therefore, combining  \eqref{eq2.10}, \eqref{eq2.11} and \eqref{eq2.12}, we get \eqref{eq2.6}. The proof is completed. \end{proof}

\begin{Lemma}\label{Le2.4}
For $\alpha>m+k-1$, the following multiple integration exists and
\begin{equation*}
\int_{0}^1dx_{m}\cdots\int_{0}^{1-\sum\limits_{i=2}^{m}x_{i}}\bigg(1-\sum\limits_{i=1}^{m}x_{i}\bigg)^{\alpha-m-k}\prod\limits_{i=1}^{m}
x_{i}^{q_{i}}dx_{1}
=\frac{\prod\limits_{i=1}^{m}\Gamma(q_{i}+1)\Gamma(\alpha-m-k+1)}{\Gamma(\alpha+\sum_{i=1}^{m}q_{i}-k+1)},
\end{equation*}
where $q=(q_{1},\cdots,q_{m})\in(\mathbb{R}_{+})^{m}$, here $\mathbb{R}_{+}$ denotes the set of positive real numbers.
\end{Lemma}
\begin{proof}[Proof]
The proof is trivial, we omit it.
\end{proof}

\begin{Lemma}\label{Le2.5}
For any $p\in \mathbb{N}^{n}$, $q\in\mathbb{N}^{m}$ and $\alpha>m+n$, we have
\begin{equation}\label{eq2.13}
{\Vert z^{p}w^{q}\Vert}^2=\frac{\prod\limits_{i=1}^{n}\Gamma(p_{i}+1)\prod\limits_{i=1}^{m}\Gamma(q_{i}+1)}{[\mu((\nu+1)\alpha+\vert q\vert)]^{\vert p\vert}\chi(\alpha,\vert q\vert)},
\end{equation}
where $w^q$, $|q|$, $\chi(\alpha,\vert q\vert)$ and ${\Vert z^{p}w^{q}\Vert}^2$ are given by
\begin{equation*}
w^q:=\prod_{j=1}^mw_j^{q_j},\;\;|q|=\sum_{j=1}^mq_j,
\end{equation*}
\begin{equation}\label{eq11}
\chi(\alpha,\vert q\vert):=\frac{[(\nu+1)\alpha+\vert q\vert]^{n}}{\sum\limits_{d=0}^{n}{n\choose d}\nu^{n-d}\frac{\Gamma(\alpha-m-d)}{\Gamma(\alpha-d+\vert q\vert)}}
\end{equation}
and
\begin{equation}\label{eq2.14}
{\Vert z^{p}w^{q}\Vert}^2:=\int_{D_{n,m}(\mu)}{\vert z^{p}w^{q}\vert}^{2}\exp\{-\alpha\Phi\}\frac{\omega(\mu;\nu)^{m+n}}{(n+m)!}
\end{equation}
for $ w:=(w_1,\ldots,w_m),\;q:=(q_1,\ldots,q_m).$
\end{Lemma}

\begin{proof}[Proof]
Firstly, it is well known that
$$\frac{\big(\frac{\sqrt{-1}}{2\pi}\partial\overline{\partial}\Phi\big)^{n+m}}{(n+m)!}=\det\bigg(\frac{\partial^{2}\Phi}{\partial Z_{i}\partial\overline{Z_{j}}}\bigg)\frac{\omega_{0}^{n+m}}{(n+m)!},$$
where $\omega_{0}=\frac{\sqrt{-1}}{2\pi}\sum_{i=1}^{n+m}dZ_{i}\wedge d\overline{Z_{i}}$. Therefore, by Lemma \ref{Le2.3}, we obtain
\begin{equation*}
{\Vert z^{p}w^{q}\Vert}^2=\frac{\mu^n}{\pi^{n+m}}\int_{D_{n,m}(\mu)}\vert z\vert^{2p}\vert w\vert^{2q}e^{-\mu((\nu+1)\alpha-m){\Vert z\Vert}^2}(1-{\Vert \widetilde{w}\Vert}^{2})^{\alpha-m-1}\big[\nu+
(1-{\Vert \widetilde{w}\Vert}^{2})^{-1}\big]^{n}dm(z)dm(w),
\end{equation*}
where $dm$ is the Euclidean measure. Thus, by using the polar coordinates (namely, $z_{j}=r_{j}e^{i\theta_{j}}$, $w_{l}=k_{l}e^{i\theta_{l}}$), it follows
\begin{equation*}
{\Vert z^{p}w^{q}\Vert}^2={2^{n+m}\mu^n}\int_{\Vert k\Vert^{2}<e^{-\mu\Vert r\Vert^{2}}\atop r\geq 0,k\geq 0}r^{2p+1}k^{2q+1}e^{-\mu((\nu+1)\alpha-m)\Vert r\Vert^{2}}(1-\Vert \widetilde{k}\Vert^{2})^{\alpha-m-1}\big[\nu+
(1-{\Vert \widetilde{k}\Vert}^{2})^{-1}\big]^{n}drdk,
\end{equation*}
where $\widetilde{k}=e^{\frac{\mu}{2}\Vert r\Vert^{2}}k$, $r=(r_{1},\cdots,r_{n})$ and $k=(k_{1},\cdots,k_{m})$. Therefore, by setting $s_{i}=r_{i}^2$ ($1\leq i\leq n$) and $t_{j}=k_{j}^{2}$ ($1\leq j\leq m$), we have
\begin{equation*}
{\Vert z^{p}w^{q}\Vert}^2=\mu^{n}\int_{\sum\limits_{i=1}^mt_{i}<e^{-\mu\sum\limits_{i=1}^ns_{i}}\atop t_i\geq 0,s_i\geq 0}s^pt^qe^{-\mu((\nu+1)\alpha-m)\sum\limits_{i=1}^ns_{i}}(1- \sum\limits_{i=1}^m\widetilde{t_{i}})^{\alpha-m-1}\big[\nu+
(1-\sum\limits_{i=1}^m\widetilde{t_{i}})^{-1}\big]^{n}dsdt,
\end{equation*}
where $\widetilde{t_{i}}=e^{\mu\sum\limits_{i=1}^ns_{i}}t_{i}$, and so it follows
\begin{equation*}
{\Vert z^{p}w^{q}\Vert}^2=\mu^{n}\int_{(\mathbb{R}_{+})^n}s^{p}e^{-\mu((\nu+1)\alpha+\vert q\vert)\sum\limits_{i=1}^ns_{i}}ds\int_{\sum\limits_{i=1}^m\widetilde{t_{i}}<1\atop \widetilde{t_{i}}\geq 0}\sum\limits_{d=0}^{n}{n\choose d}\nu^{n-d}(1-
\sum\limits_{i=1}^m\widetilde{t_{i}})^{\alpha-m-1-d}\widetilde{t}^qd\widetilde{t}.
\end{equation*}

Since $\alpha>m+n$, by Lemma \ref{Le2.4}, we have
\begin{equation*}
{\Vert z^{p}w^{q}\Vert}^2=\mu^n\sum\limits_{d=0}^{n}{n\choose d}\nu^{n-d}\frac{\prod\limits_{i=1}^{m}\Gamma(q_{i}+1)\Gamma(\alpha-m-d)}{\Gamma(\alpha-d+\vert q\vert)}\int_{(\mathbb{R}_{+})^n}s^pe^{-\mu((\nu+1)\alpha+\vert
q\vert)\sum\limits_{i=1}^ns_{i}}ds.
\end{equation*}
By the definition of Gamma functions, we obtain
\begin{equation*}
{\Vert z^{p}w^{q}\Vert}^2=\mu^n\sum\limits_{d=0}^{n}{n\choose d}\nu^{n-d}\frac{\prod\limits_{i=1}^{n}\Gamma(p_{i}+1)\prod\limits_{i=1}^{m}\Gamma(q_{i}+1)\Gamma(\alpha-m-d)}{\Gamma(\alpha-d+\vert q\vert)\big[\mu\big((\nu+1)\alpha+\vert
q\vert\big)\big]^{\vert p\vert+n}}.
\end{equation*}
The proof is completed. \end{proof}

\begin{Theorem}
Let $(D_{n,m}(\mu),g(\mu;\nu))$ be the Fock-Bargmann-Hartogs domain $D_{n,m}(\mu)$ endowed with the metric $g(\mu;\nu)$. Then, for $\alpha>m+n$, the Bergman kernel of the weight Hilbert space $\mathcal{H}_{\alpha}$  defined by
\begin{equation*}
\mathcal{H}_{\alpha}:=\bigg\{f\in \mathrm{Hol}(D_{n,m}(\mu)):\int_{D_{n,m}(\mu)}\vert f\vert^{2}\exp\{-\alpha \Phi\}\frac{\omega(\mu;\nu)^{m+n}}{(n+m)!}<\infty\bigg\}
\end{equation*}
can be expressed as follows
\begin{equation}\label{eq2.15}
K_{\alpha}(z,w,\overline{z},\overline{w})=(\alpha-m-n)_{m+n}e^{\mu(\nu+1)\alpha\Vert z\Vert^2}\sum\limits_{q\in\mathbb{N}^m}\psi(\alpha,\vert q\vert)\frac{\Gamma(\vert q
\vert+\alpha)}{\Gamma(\alpha)\prod\limits_{i=1}^{m}\Gamma(q_{i}+1)}\widetilde{w}^q\overline{\widetilde{w}^q},
\end{equation}
where $(\alpha)_k$ and $\psi(\alpha,\vert q\vert)$ are defined by
\begin{equation*}
  (\alpha)_k:=\frac{\Gamma(\alpha+k)}{\Gamma(\alpha)}=\alpha(\alpha+1)\cdots(\alpha+k-1)
\end{equation*}
and
\begin{equation}\label{eq2.16}
\psi(\alpha,\vert q\vert):=\frac{\Gamma(\alpha-m-n)\chi(\alpha,\vert q\vert)}{\Gamma(\alpha+\vert q\vert)}.
\end{equation}
\end{Theorem}

\begin{proof}[Proof]
Since $\big\{\frac{z^pw^q}{\Vert z^pw^q\Vert}\big\}$ constitute an orthonormal basis of  $\mathcal{H}_{\alpha}$, we have
\begin{equation}
K_{\alpha}(z,w,\overline{z},\overline{w})=\sum_{p\in\mathbb{N}^m,q\in\mathbb{N}^m}\frac{z^pw^q\overline{z^pw^q}}{\Vert z^pw^q\Vert^2}
\end{equation}
by the Fock-Bargmann-Hartogs domain $D_{n,m}(\mu)$ being a Reinhardt domain. The formula \eqref{eq2.13} implies that
\begin{equation*}
\begin{aligned}
&K_{\alpha}(z,w,\overline{z},\overline{w})\\
=&\sum_{p\in\mathbb{N}^m,q\in\mathbb{N}^m}\frac{[\mu((\nu+1)\alpha+\vert q\vert)]^{\vert p\vert}\chi(\alpha,\vert q\vert)}{\prod\limits_{i=1}^{n}\Gamma(p_{i}+1)\prod\limits_{i=1}^{m}\Gamma(q_{i}+1)}z^pw^q\overline{z^pw^q}\\
=&\sum\limits_{q\in\mathbb{N}^m}e^{\mu((\nu+1)\alpha+\vert q\vert)\Vert z\Vert^2}\frac{\chi(\alpha,\vert q\vert)}{\prod\limits_{i=1}^{m}\Gamma(q_{i}+1)}w^q\overline{w^q}\\
=&e^{\mu(\nu+1)\alpha\Vert z\Vert^2}\sum\limits_{q\in\mathbb{N}^m}\frac{\chi(\alpha,\vert q\vert)}{\prod\limits_{i=1}^{m}\Gamma(q_{i}+1)}{ \widetilde{w}^q\overline{\widetilde{w}^q}},
\end{aligned}
\end{equation*}
where $\widetilde{w}$ is defined by \eqref{eq2.7}. By simplifying the above formula, we have
\begin{equation*}
\begin{aligned}
&K_{\alpha}(z,w,\overline{z},\overline{w})\\
=&(\alpha-m-n)_{m+n}e^{\mu(\nu+1)\alpha\Vert z\Vert^2}\sum\limits_{q\in\mathbb{N}^m}\frac{\Gamma(\alpha-m-n)\chi(\alpha,\vert q\vert)}{\Gamma(\alpha+\vert q\vert)}\frac{\Gamma(\alpha+\vert
q\vert)}{\Gamma(\alpha)\prod\limits_{i=1}^{m}\Gamma(q_{i}+1)}\widetilde{w}^q\overline{\widetilde{w}^q}\\
=&(\alpha-m-n)_{m+n}e^{\mu(\nu+1)\alpha\Vert z\Vert^2}\sum\limits_{q\in\mathbb{N}^m}\psi(\alpha,\vert q\vert)\frac{\Gamma(\vert q \vert+\alpha)}{\Gamma(\alpha)\prod\limits_{i=1}^{m}\Gamma(q_{i}+1)}\widetilde{w}^q\overline{\widetilde{w}^q}.
\end{aligned}
\end{equation*}
So we finish the proof.
\end{proof}

Now we give the explicit expression of the Rawnsley's $\varepsilon$-function of the Fock-Bargmann-Hartogs domain $(D_{n,m}(\mu),g(\mu;\nu))$ as follows.

\begin{Theorem}
Suppose $(D_{n,m}(\mu),g(\mu;\nu))$ is the Fock-Bargmann-Hartogs domain $D_{n,m}(\mu)$ endowed with the metric $g(\mu;\nu)$. Then the explicit expression of the Rawnsley's $\varepsilon$-function of $(D_{n,m}(\mu),g(\mu;\nu))$ is given by
\begin{equation}\label{eq2.18}
\varepsilon_{(\alpha,g(\mu;\nu))}=(\alpha-m-n)_{m+n}(1-\Vert \widetilde{w}\Vert^{2})^{\alpha}\sum\limits_{q\in\mathbb{N}^m}\psi(\alpha,\vert q\vert)\frac{\Gamma(\vert q
\vert+\alpha)}{\Gamma(\alpha)\prod\limits_{i=1}^{m}\Gamma(q_{i}+1)}\widetilde{w}^q\overline{\widetilde{w}^q}.
\end{equation}
\end{Theorem}

\begin{proof}[Proof]
In fact, by the definition \eqref{eq1.1}, we have
\begin{equation*}
\varepsilon_{(\alpha,g(\mu;\nu))}=e^{-\alpha\Phi(z,w)}K_{\alpha}(z,w,\overline{z},\overline{w}),
\end{equation*}
and by the definition \eqref{eq1.3}, we get
 \begin{equation*}
 e^{-\alpha\Phi(z,w)}=e^{-\mu(\nu+1)\alpha\Vert z\Vert^2}(1-\Vert \widetilde{w}\Vert^{2})^{\alpha}.
 \end{equation*}
Therefore, by \eqref{eq2.15}, we obtain
\begin{equation}
\varepsilon_{(\alpha,g(\mu;\nu))}=(\alpha-m-n)_{m+n}(1-\Vert \widetilde{w}\Vert^{2})^{\alpha}\sum\limits_{q\in\mathbb{N}^m}\psi(\alpha,\vert q\vert)\frac{\Gamma(\vert q
\vert+\alpha)}{\Gamma(\alpha)\prod\limits_{i=1}^{m}\Gamma(q_{i}+1)}\widetilde{w}^q\overline{\widetilde{w}^q}.
\end{equation}
The proof is finished.
\end{proof}

\setcounter{equation}{0}
\section{The proof of main results}
We at first give the following lemma.

\begin{Lemma}[see D'Angelo \cite{D2} Lemma 2]\label{Le3.1}
Let $x=(x_{1},\cdots,x_{m})\in \mathbb{R}^{m}$  with $\Vert x\Vert<1$ and $s\in \mathbb{R}$ with $s>0$. Then
\begin{equation}
\sum\limits_{q\in\mathbb{N}^m}\frac{\Gamma(\vert q \vert+s)}{\Gamma(s)\prod\limits_{i=1}^{m}\Gamma(q_{i}+1)}x^{2q}=\frac{1}{(1-{\Vert x\Vert^2})^{s}}.
\end{equation}
\end{Lemma}

Now we give the proof of our main result.

\begin{proof}[The proof of Theorem \ref{Thm1.1}]
Firstly, we must ensure that the weighted Hilbert space $\mathcal{H}_{\alpha}$  will not reduce to zero subspace. In our case, it is easy to see
\begin{equation*}
\mathcal{H}_{\alpha}\neq\{0\}\Longleftrightarrow \alpha>m+n.
\end{equation*}

In fact, if $\alpha>m+n$, then by Lemma \ref{Le2.5}, $\mathcal{H}_{\alpha}\neq\{0\}$. Conversely, if $\mathcal{H}_{\alpha}\neq\{0\}$,  we assume $f(z,w)\in \mathcal{H}_{\alpha}\setminus \{0\}$. Since $D_{n,m}(\mu)$ is a complete Reinhardt domain,  $f(z,w)$ can be expressed by
$$f(z,w)=\sum_{\beta}f_{\beta}(z)w^{\beta},$$
where the series is uniformly convergent on any compact subset of $D_{n,m}(\mu)$ and every $f_{\beta}$ is holomorphic on $\mathbb{C}^{n}$. So we have
\begin{equation*}
\Vert f\Vert^{2}=\sum_{\beta}\frac{\mu^n}{\pi^{n+m}}\int_{D_{n,m}(\mu)}\vert f_{\beta}\vert^{2}\vert w\vert^{2\beta}e^{-\mu((\nu+1)\alpha-m){\Vert z\Vert}^2}(1-{\Vert \widetilde{w}\Vert}^{2})^{\alpha-m-1}\big[\nu+
(1-{\Vert \widetilde{w}\Vert}^{2})^{-1}\big]^{n}dV<+\infty.
\end{equation*}
Note $\Vert \widetilde{w}\Vert= \Vert  e^{\frac{\mu}{2}\Vert z\Vert^2}w\Vert <1$ on $D_{n,m}(\mu)$ by definition. Thus, for each $\beta$, we have
$$\frac{\mu^n}{\pi^{n+m}}\int_{D_{n,m}(\mu)}\vert f_{\beta}\vert^{2}\vert w\vert^{2\beta}e^{-\mu((\nu+1)\alpha-m){\Vert z\Vert}^2}(1-{\Vert \widetilde{w}\Vert}^{2})^{\alpha-m-1}\big[\nu+
(1-{\Vert \widetilde{w}\Vert}^{2})^{-1}\big]^{n}dV \leq \Vert f\Vert^{2} <+\infty.$$
Assume  $f_{\beta_0}\not\equiv 0$ on $\mathbb{C}^{n}$. Thus, by Fubini's Th., we have
\begin{equation*}
\sum\limits_{d=1}^{n}{n\choose d}\nu^{n-d}\int_{{\Vert \widetilde{w}\Vert}^{2}<1}\vert \widetilde{w}\vert^{2\beta_0}(1-{\Vert \widetilde{w}\Vert}^{2})^{\alpha-m-1-d}d\widetilde{w}<+\infty.
\end{equation*}
Hence we have $\alpha-m-1-d>-1$ for all $1\leq d\leq n$. Therefore $\alpha>m+n$.

Secondly, from \eqref{eq2.18}, we know that $\varepsilon_{(\alpha,g(\mu;\nu))}(z,w)$ is independent of $(z,w)$ if and only if there exists a constant $\lambda$ with respect to $(z,w)$ such that
\begin{equation}
(1-{\Vert \widetilde{w}\Vert^2})^{-\alpha}=\lambda\sum\limits_{q\in\mathbb{N}^m}\psi(\alpha,\vert q\vert)\frac{\Gamma(\vert q \vert+\alpha)}{\Gamma(\alpha)\prod\limits_{i=1}^{m}\Gamma(q_{i}+1)}\widetilde{w}^q\overline{\widetilde{w}^q}.
\end{equation}
Thus, by Lemma \ref{Le3.1}, we get that $\psi(\alpha,\vert q\vert)$ is a constant with respect to $|q|$. From \eqref{eq11} and \eqref{eq2.16}, we get
\begin{equation}\label{eq3.3}
\psi(\alpha,\vert q\vert)=\frac{[(\nu+1)\alpha+\vert q\vert]^{n}}{\sum\limits_{d=0}^{n}{n\choose d}\nu^{n-d}(\alpha-m-n)_{n-d}(\alpha-d+\vert q\vert)_{d}},
\end{equation}
and, obviously, $\psi(\alpha,\vert q\vert)$ tends to $1$ as $\vert q\vert\rightarrow\infty$. Thus $\varepsilon_{(\alpha,g(\mu;\nu))}(z,w)$ is a constant if and only if for any $x,y\in\mathbb{R}$,
\begin{equation}\label{eq3.4}
[(\nu+1)x+y]^{n}=\sum\limits_{d=0}^{n}{n\choose d}\nu^{n-d}(x-m-n)_{n-d}(x-d+y)_{d}.
\end{equation}

Now we claim that \eqref{eq3.4} holds if and only if
\begin{equation*}
n=1,\quad \nu=-\frac{1}{m+1}.
\end{equation*}

Indeed, if \eqref{eq3.4} holds, then by setting $x+y=1$ in \eqref{eq3.4}, we have
\begin{equation}\label{eq3.5}
(\nu x+1)^{n}=\nu^{n}(x-m-n)_{n}.
\end{equation}
If $\nu=0$, it is impossible that \eqref{eq3.5} holds. Thus we assume $\nu\neq0$. Then \eqref{eq3.5} implies
\begin{equation}
\left(x+\frac{1}{\nu}\right)^{n}=\prod\limits_{j=1}^{n}(x-m-j).
\end{equation}
Since the right side of the above formula has no multiple divisor, we have $n=1$, and so
$$-\frac{1}{\nu}=m+1\Rightarrow\nu=-\frac{1}{m+1}.$$

On the other hand, it is easy to see that \eqref{eq3.4} holds when $n=1$ and $\nu=-\frac{1}{m+1}$. The proof is finished.
\end{proof}

 \noindent\textbf{Acknowledgments}
 E. Bi was supported by the Fundamental Reasearch Funds for the Central Universities (No.2014201020204), Z. Feng was supported by the
Scientific Research Fund of Sichuan Provincial Education Department (No.15ZA0284), and Z. Tu was supported by the National Natural
Science Foundation of China (No.11271291).


\addcontentsline{toc}{section}{References}
\phantomsection
\renewcommand\refname{References}
\small{
}

\begin{thebibliography}{99}
\setlength{\parskip}{0pt}
\bibitem{Are-Loi} Arezzo, C., Loi, A.: Moment maps, scalar curvature and quantization of K¡§ahler manifolds. Comm. Math. Phys. \textbf{243}, 543-559 (2004)


\bibitem{Berezin}Berezin, F.A.: Quantization, Math. USSR Izvestiya. \textbf{8}, 1109-1163 (1974)

\bibitem{CGR}Cahen, M., Gutt, S., Rawnsley, J.: Quantization of K\"{a}hler manifolds. I: Geometric interpretation of Berezin's quantization.
J. Geom. Phys. \textbf{7}, 45-62 (1990)


\bibitem{Cat}Catlin, D.: The Bergman kernel and a theorem of Tian.
Analysis and geometry in several complex variables (Katata, 1997),
Trends Math., Birkh\"{a}user Boston, Boston, MA, pp. 1-23 (1999)


\bibitem{D2}D'Angelo, J.P.: An explicit computation of the Bergman kernel
function. J. Geom. Anal. \textbf{4}(1), 23-34 (1994)

\bibitem{Donaldson}Donaldson, S.: Scalar curvature and projective
embeddings, I.  J. Differential Geom. \textbf{59}, 479-522 (2001)

\bibitem{E0}Engli\v{s}, M.: Berezin Quantization and Reproducing Kernels on Complex Domains,
Trans. Amer. Math. Soc.  \textbf{348}, 411-479 (1996)

\bibitem{E1}Engli\v{s}, M.: A Forelli-Rudin construction and asymptotics of weighted Bergman kernels. J. Funct. Anal. \textbf{177}, 257-281 (2000)

\bibitem{E2}Engli\v{s}, M.: The asymptotics of a Laplace integral on a K\"{a}hler manifold. J. Reine Angew. Math. \textbf{528}, 1-39 (2000)

\bibitem{E3}Engli\v{s}, M.: Weighted Bergman kernels and balanced metrics.  RIMS Kokyuroku \textbf{1487}, 40-54(2006)







\bibitem{FT}Feng, Z.M., Tu, Z.H.: On canonical metrics on Cartan-Hartogs domains. Math. Z. \textbf{278}, 301-320 (2014)

\bibitem{FT2}Feng, Z.M., Tu, Z.H.: Balanced  metrics on some Hartogs type domains over bounded
symmetric domains. Ann. Glob. Anal. Geom. \textbf{47}, 305-333 (2015)







\bibitem{Kim}Kim, H., Ninh, V.T., Yamamori, A.: The automorphism group of a certain
unbounded non-hyperbolic domain. J. Math. Anal. Appl. \textbf{409}(2),
637-642 (2014)


\bibitem{KY} Kim, H., Yamamori, A.:  An application of a Diederich-Ohsawa theorem in characterizing some Hartogs domains. Bull. Sci. Math. \textbf{139}(7), 737-749 (2015)




\bibitem{Lo}Loi, A.: The Tian-Yau-Zelditch asymptotic expansion for
real analytic K\"{a}hler metrics. Int. J Geom. Methods in Modern
Phys. \textbf{1}, 253-263 (2004)

\bibitem{Loi2}Loi, A.: Bergman and balanced metrics on complex manifolds. Int. J. Geom. Methods Mod. Phys. \textbf{02}, 553 (2005)

\bibitem{Loi-Mossa}Loi, A.,  Mossa, R.: Berezin quantization of homogeneous bounded domains. Geometriae Dedicata. \textbf{161}(1), 119-128 (2012)

\bibitem{Loi-Zed}Loi, A., Zedda, M.: Balanced metrics on Hartogs domains.  Abhandlungen aus dem Mathematischen Seminar der Universit\"{a}t Hamburg. \textbf{81}(1), 69-77 (2011)

\bibitem{LZ}Loi, A., Zedda, M.: Balanced metrics on Cartan and Cartan-Hartogs domains. Math. Z. \textbf{270}, 1077-1087 (2012)

\bibitem{Loi-Zed-Zud} Loi, A., Zedda, M.,  Zuddas, F.: Some remarks on the K\"{a}hler geometry of the Taub-NUT metrics. Annals of Global Analysis and Geometry.
 \textbf{41}(4), 515-533 (2012)





\bibitem{MM07}Ma, X., Marinescu, G.: Holomorphic Morse
inequalities and Bergman kernels. Progress in Mathematics, Vol. 254,
Birkh\v{a}user Boston Inc., Boston, MA (2007)

\bibitem{MM08}Ma, X., Marinescu, G.: Generalized Bergman kernels on symplectic
manifolds. Adv. Math. \textbf{217}(4), 1756-1815 (2008)

\bibitem{MM12}Ma, X., Marinescu, G.: Berezin-Toeplitz
quantization on K\"{a}hler manifolds. J. reine angew. Math.
\textbf{662}, 1-56 (2012)





\bibitem{TW}Tu, Z.H., Wang, L.: Rigidity of proper holomorphic mappings between certain unbounded non-hyperbolic domians. J. Math. Anal. Appl. \textbf{419}, 703-714 (2014)









\bibitem{Zed}Zedda, M.: Canonical metrics on Cartan-Hartogs domains. International Journal of Geometric Methods in Modern Physics. \textbf{9}(1), (2012)


\bibitem{Zeld}Zelditch, S.: Szeg\"{o} kernels and a theorem of Tian. Internat. Math. Res. Notices. \textbf{6}, 317-331 (1998)


\end{thebibliography}
\end{document}